\documentclass[amscd,amssymb,verbatim]{amsart}
\usepackage{amssymb,amsfonts,amsmath,amscd}
\usepackage[all]{xy}

\theoremstyle{plain}
\newtheorem{Thm}{Theorem}[section]

\newtheorem{Prop}[Thm]{Proposition}
\newtheorem{Cor}[Thm]{Corollary}

\theoremstyle{definition}
\newtheorem{Def}[Thm]{Definition}
\newtheorem{Ex}[Thm]{Example}

\theoremstyle{remark}

\newcommand{\B}{{\mathcal L}}
\newcommand{\C}{\mathbb{C}}
\newcommand{\cC}{{\mathcal C}}
\newcommand{\bbD}{\mathbb{D}}

\newcommand{\G}{{\mathcal G}}
\newcommand{\cH}{{\mathcal H}}

\newcommand{\cK}{{\mathcal K}}
\newcommand{\bbF}{\mathbb{F}}

\newcommand{\R}{\mathbb{R}}
\newcommand{\T}{\mathbb{T}}

\newcommand{\cT}{{\mathcal T}}

\newcommand{\Hom}{{\mathrm {Hom}}}

\newcommand{\ua}{\uparrow}
\newcommand{\XYZ}{X^{Z\ua Y}}

\begin{document}
\title[Homotopy Lifting Theorem]{The Homotopy Lifting Theorem for Semiprojective C*-Algebras}
\author{Bruce Blackadar}
\address{Department of Mathematics/0084 \\ University of Nevada, Reno \\ Reno, NV 89557, USA}
\email{bruceb@unr.edu}

\date{\today}

\maketitle
\begin{abstract}
We prove a complete analog of the Borsuk Homotopy Extension Theorem for arbitrary
semiprojective C*-algebras.  We also obtain some other results about
semiprojective C*-algebras: a partial lifting theorem with specified quotient,
a lifting result for homomorphisms close to a liftable homomorphism, and that sufficiently close homomorphisms
from a semiprojective C*-algebra are homotopic.

\end{abstract}

\section{Introduction}

It seems obligatory in any exposition of the theory of Absolute Neighborhood Retracts (ANR's) in topology to refer to the
Borsuk Homotopy Extension Theorem as ``one of the most important results in the theory of ANR's'' (as well it is).

\begin{Thm} {\sc [Borsuk Homotopy Extension Theorem]} \cite{BorsukProlongements}, \cite[8.1]{BorsukRetracts}\label{BorsukHom}
Let $X$ be an ANR, $Y$ a compact metrizable space, $Z$ a closed subspace of $Y$, $(\phi_t)$ ($0\leq t\leq 1$) a uniformly continuous path of
continuous functions from $Z$ to $X$ (i.e.\ $h(t,z)=\phi_t(z)$ is a homotopy from $\phi_0$ to $\phi_1$).  Suppose $\phi_0$
extends to a continuous function $\bar\phi_0$ from $Y$ to $X$.  Then there is a uniformly continuous path $\bar\phi_t$
of extensions of the $\phi_t$ to functions from $Y$ to $X$ (i.e.\ $\bar h(t,y)=\bar\phi_t(y)$ is a homotopy from $\bar\phi_0$
to $\bar\phi_1$).
\end{Thm}

In particular, any function from $Z$ to $X$ homotopic to an extendible function is extendible.  The theorem also works for
metrizable spaces which are not necessarily compact when phrased in the homotopy language; we have stated it in the version
which can potentially be extended to noncommutative C*-algebras.  The theorem can be regarded as giving a ``universal cofibration property''
for maps into ANR's.

\smallskip

There is a direct analog of (compact) ANR's in the category of (separable) noncommutative C*-algebras: the semiprojective C*-algebras
(\cite{BlackadarShape}, \cite[II.8.3.7]{BlackadarOperator}).  Many results about ANR's carry through to semiprojective C*-algebras
with essentially identical proofs (just ``turning arrows around'').  However, Borsuk's proof of the Homotopy Extension Theorem
is not one of these: the proof simply does not work in the noncommutative case.  The underlying reason is that in a metrizable
space, every closed set is a $G_\delta$, but this is false in the primitive ideal space of a separable noncommutative C*-algebra in general.

We can, however, by a different argument obtain a complete analog of the Borsuk Homotopy Extension Theorem for arbitrary
semiprojective C*-algebras (\ref{HomExt}).  In the course of the proof we obtain some other results about
semiprojective C*-algebras which are of interest: a partial lifting theorem with specified quotient (\ref{SpecLift}),
a lifting result for homomorphisms close to a liftable homomorphism (\ref{CloseExtend}), and that sufficiently close homomorphisms
from a semiprojective C*-algebra are homotopic (\ref{CloseHom}).



\section{The General Chinese Remainder Theorem}

We will make use of a general ``folklore'' result from ring theory, which can be called the Generalized Chinese Remainder Theorem.
Although this result should probably be one of the standard isomorphism theorems for rings, it is not covered in most algebra texts,
so we give the simple proof.  A variant can be found in \cite[Prop.\ 3.1]{PedersenPullback}, with the same proof.

\begin{Prop}\label{Chinese}
Let $R$ be a ring, and $I$ and $J$ (two-sided) ideals in $R$.  Then the map $\phi:a\mapsto(a\mod I,a\mod J)$ gives an
isomorphism from $R/(I\cap J)$ onto the fibered product
$$P=(R/I)\oplus_{(\pi_1,\pi_2)}(R/J)=\{(x,y) : x\in R/I, y\in R/J, \pi_1(x)=\pi_2(y)\}\subseteq(R/I)\oplus(R/J)$$
where $\pi_1:R/I\to R/(I+J)$ and $\pi_2:R/J\to R/(I+J)$ are the quotient maps.  ($P$ is the pullback of $(\pi_1,\pi_2)$.)
\end{Prop}

\begin{proof}
It is obvious that $\phi$ (regarded as a map from $R$ to $P$) is a homomorphism with kernel $I\cap J$.  We need only show that $\phi$ is surjective.
Let $(x,y)\in P$.  Write $\pi_I$ and $\pi_J$ for the quotient maps from $R$ to $R/I$ and $R/J$ respectively.
Then there is a $b\in R$ with $\pi_J(b)=y$.  We have
$$\pi_1(x-\pi_I(b))=\pi_1(x)-\pi_1(\pi_I(b))=\pi_1(x)-\pi_2(y)=0$$
and the kernel of $\pi_1$ is exactly $\pi_I(J)$,
so there is a $c\in J$ with $\pi_I(c)=x-\pi_I(b)$.  Set $a=b+c$.  Then $\pi_I(a)=x$ and $\pi_J(a)=\pi_J(b)+\pi_J(c)=y$.
Thus $\phi(a)=(x,y)$.
\end{proof}

In particular, to define a homomorphism from another ring into $R/(I\cap J)$, it suffices to give a compatible pair of
homomorphisms into $R/I$ and $R/J$.

\smallskip

There is, of course, a version of this result for finitely many ideals, but it is somewhat complicated to state.  The
usual Chinese Remainder Theorem is the special case where $I+J=R$; the fibered product is then just the full direct sum.

To apply this result to C*-algebras, note that if $I$ and $J$ are closed ideals in a C*-algebra, then $I+J$ is
also closed (see e.g.\ \cite[II.5.1.3]{BlackadarOperator}).  One can replace ``homomorphism'' with ``*-homomorphism'' throughout.
(When working with C*-algebras, we will take ``homomorphism'' to mean ``*-homomorphism.'')

\section{Partial Liftings with Specified Quotient}

Recall the definition of semiprojectivity (\cite{BlackadarShape}, \cite[II.8.3.7]{BlackadarOperator}):
A separable C*-algebra $A$ is {\em semiprojective} if, whenever $B$ is a C*-algebra, $(J_n)$ an increasing sequence of closed (two-sided) ideals of $B$, and $J=[\cup J_n]^-$, then any homomorphism $\phi:A\to B/J$ can be partially lifted to a homomorphism
$\psi:A\to B/J_n$ for some sufficiently large $n$.
But suppose in the above situation, with $A$ semiprojective, we also have another closed ideal $I$ of $B$ and a homomorphism $\tilde \phi$ from $A$ to $B/I$ such that
$\phi$ and $\tilde\phi$ agree mod $I+J$.  Can we partially lift $\phi$ to $\psi$ so that $\psi$ agrees with $\tilde\phi$
mod $I+J_n$?  The next result shows that this is always possible.  For any closed ideal $K$ of $B$, write $\pi_K$
for the quotient map to $B/K$ (by slight abuse of notation, this same symbol will be used for the quotient map from $B/L$ to $B/K$
for any closed ideal $L$ contained in $K$.)

\begin{Thm} {\sc [Specified Quotient Partial Lifting Theorem]}\label{SpecLift}
Let $A$ be a semiprojective C*-algebra, $B$ a C*-algebra, $(J_n)$ an increasing sequence of closed ideals of $B$
with $J=[\cup J_n]^-$, $I$ another closed ideal of $B$, and $\phi:A\to B/J$ and $\tilde\phi:A\to B/I$ *-homomorphisms
with $\pi_{I+J}\circ\phi=\pi_{I+J}\circ\tilde\phi$.
Then for some sufficiently large $n$ there is a *-homomorphism $\psi:A\to B/J_n$ such that $\pi_J\circ\psi=\phi$ and
$\pi_{I+J_n}\circ\psi=\pi_{I+J_n}\circ\tilde\phi$.
\end{Thm}

Pictorially, we have the following diagram which can be made to commute:

\[
\xymatrix @=1pc @*[c] {
 & \  & \  & B \ar[dd] \ar[rd] &
\\ & & & \ar[r] & B/I \ar[dd]
\\ & & & B/J_1 \ar[dd] \ar[rd] &
\\ & & & & B/(I+J_1) \ar[dd]
\\ & & & B/J_2 \ar[dd] \ar[rd] &
\\ & & & & B/(I+J_2) \ar[dd]
\\ & & & \vdots \ar[dd] &
\\ & & & & \vdots \ar[dd]
\\ & & & B/J_n \ar[dd] \ar[rd] &
\\ & & & & B/(I+J_n) \ar[dd]
\\ & & & \vdots \ar[dd] &
\\ & & & & \vdots \ar[dd]
\\ A \ar@(u,l)@{-}^{\tilde\phi}[rrruuuuuuuuuuu] \ar@{.>}_{\psi}[rrruuuu] \ar^{\phi}[rrr] & & & B/J \ar[rd] &
\\ & & & & B/(I+J) }\]

\begin{proof}
It is obvious that $\cup_n(I+J_n)$ is dense in $I+J$.  It is not obvious that $\cup_n(I\cap J_n)$ is dense in $I\cap J$,
but this can be proved using \cite[II.5.1.3]{BlackadarOperator}: if $x\in I\cap J$, then
$$0=\inf_n [\inf_{y\in J_n}\|x-y\| ]=\inf_n [\inf_{z\in I\cap J_n}\|x-z\| ]\ .$$
By \ref{Chinese} $\phi$ and $\tilde\phi$ define a homomorphism $\bar\phi$ from $A$ to $B/(I\cap J)$, which
partially lifts to a homomorphism $\bar\psi$ from $A$ to $B/(I\cap J_n)$ for some $n$ by semiprojectivity.  The map $\bar\psi$
defines compatible homomorphisms $\psi:A\to B/J_n$ and $\tilde\psi:A\to B/I$.   Then 
$$\tilde\psi=\pi_I\circ\psi:A\to B/(I\cap J_n)\to B/I$$
$$=\pi_I\circ\pi_{I\cap J}\circ\psi:A\to B/(I\cap J_n)\to B/(I\cap J)\to B/I$$
$$=\pi_I\circ\phi:A\to B/(I\cap J)\to B/I=\tilde\phi\ .$$
Since
$$\pi_{I+J_n}\circ\psi:A\to B/J_n\to B/(I+J_n)$$
equals
$$\pi_{I+J_n}\circ\tilde\psi=\pi_{I+J_n}\circ\tilde\phi:A\to B/I\to B/(I+J_n)$$
we have that $\psi$ is the desired partial lift of $\phi$.
\end{proof}

\section{Lifting Close Homomorphisms}

If $A$ and $B$ are C*-algebras and $I$ is a closed ideal of $B$, then a homomorphism from $A$ to $B/I$ need not lift
in general to a homomorphism from $A$ to $B$, even if $A$ is semiprojective.  But suppose $\phi:A\to B/I$ does lift to $\bar\phi:A\to B$, and $\psi$ is
another homomorphism from $A$ to $B/I$ which is close to $\phi$ in the point-norm topology.  If $A$ is semiprojective, does $\psi$ also lift to $B$,
and can the lift be chosen close to $\bar\phi$ in the point-norm topology?  The answer is yes in the commutative category
\cite[3.1]{BorsukRetracts}, but the commutative proof does not generalize to the noncommutative case.  However, we can by a
different argument obtain the same result for general semiprojective C*-algebras.

\begin{Thm}\label{CloseExtend}
Let $A$ be a semiprojective C*-algebra generated by a finite or countable set $\G=\{x_1,x_2,\dots\}$ with $\lim_{j\to\infty}\|x_j\|=0$ if
$\G$ is infinite.   Then for any $\epsilon>0$ there is a $\delta>0$
such that, whenever $B$ is a C*-algebra, $I$ a closed ideal of $B$, $\phi$ and $\psi$ *-homomorphisms from $A$ to $B/I$
with $\|\phi(x_j)-\psi(x_j)\|<\delta$ for all $j$ and such that $\phi$ lifts to a *-homomorphism $\bar\phi:A\to B$ (i.e.\
$\pi_I\circ\bar\phi=\phi$), then $\psi$ also lifts to a *-homomorphism $\bar\psi:A\to B$ with $\|\bar\psi(x_j)-\bar\phi(x_j)\|<\epsilon$
for all $j$.  (The $\delta$ depends on $\epsilon$, $A$, and the set $\G$ of generators, but not on the $B$, $I$, $\phi$, $\psi$.)
\end{Thm}

\begin{proof}
Suppose the result is false.  Then there is an $\epsilon>0$ and $B_n$, $I_n$, and $\phi_n$, $\psi_n$ homomorphisms from $A$
to $B_n/I_n$ such that $\|\phi_n(x_j)-\psi_n(x_j)\|<\frac{1}{n}$ for all $j$, $\phi_n$ lifts to $\bar\phi_n:A\to B_n$,
but $\psi_n$ does not lift to any $\bar\psi_n:A\to B_n$ with $\|\bar\phi_n(x_j)-\bar\psi_n(x_j)\|<\epsilon$ for all $j$.
Let $B=\prod_n B_n$, $I=\prod_n I_n$,
$J_n$ the ideal of elements of $B$ vanishing after the $n$'th term, $J=[\cup J_n]^-=\oplus_n B_n$.
Let $\bar\phi:A\to B$ be defined by
$$\bar\phi(x)=(\bar\phi_1(x),\bar\phi_2(x),\dots)$$
and let $\phi=\pi_J\circ\bar\phi:A\to B/J$.
There is also a homomorphism $\tilde\phi$ from $A$ to $B/I\cong\prod_n(B_n/I_n)$ defined by
$$\tilde\phi(x)=(\psi_1(x),\psi_2(x),\dots)\ .$$
We have $\lim_{n\to\infty}\|\phi_n(x)-\psi_n(x)\|=0$
for all $x$ in a dense *-subalgebra of $A$ (the *-subalgebra generated by $\G$) and, since the $\phi_n$ and $\psi_n$
are uniformly bounded (all have norm 1), we have $\lim_{n\to\infty}\|\phi_n(x)-\psi_n(x)\|=0$ for all $x\in A$.  So we have that $\phi$
and $\tilde\phi$ agree mod $I+J$.
Thus by \ref{SpecLift}, for some $n$, there is a lift $\psi$ of $\phi$
to $B/J_n$ agreeing with $\tilde\phi$ mod $I+J_n$.  This lift defines $\bar\psi_k:A\to B_k$ for each $k>n$
lifting $\psi_k$.  Fix $m$ such that $\|x_j\|<\frac{\epsilon}{2}$ for all $j>m$.  Since $\psi=\phi$ mod $J$, we have
$\lim_{k\to\infty}\|\bar\phi_k(x_j)-\bar\psi_k(x_j)\|=0$ for all $j$.  Thus there is a $k$ such that
$$\|\bar\phi_k(x_j)-\bar\psi_k(x_j)\|<\epsilon$$
for $1\leq j\leq m$.  If $j>m$, we have
$$\|\bar\phi_k(x_j)-\bar\psi_k(x_j)\|\leq\|\bar\phi_k(x_j)\|+\|\bar\psi_k(x_j)\|\leq2\|x_j\|<\epsilon\ .$$
Thus $\|\bar\phi_k(x_j)-\bar\psi_k(x_j)\|<\epsilon$ for all $j$, a contradiction.

The diagram at the end of \ref{SpecLift} summarizes the construction.
\end{proof}

As in the commutative case (cf.\ \cite[IV.1.1]{HuRetracts}, \cite[4.1.1]{vanMillInfinite}), we obtain that sufficiently close homomorphisms from a semiprojective C*-algebra
are homotopic (see \cite[3.6]{BlackadarShape} for a slightly weaker version of this result with a more elementary proof):

\begin{Cor}\label{CloseHomSmall}
Let $A$ be a semiprojective C*-algebra generated by a finite or countable set $\G=\{x_1,x_2,\dots\}$ with $\lim_{j\to\infty}\|x_j\|=0$ if
$\G$ is infinite.   Then for any $\epsilon>0$ there is a $\delta>0$
such that, whenever $B$ is a C*-algebra, $\phi_0$ and $\phi_1$ *-homomorphisms from $A$ to $B$
with $\|\phi_0(x_j)-\phi_1(x_j)\|<\delta$ for all $j$, then there is a point-norm continuous path $(\phi_t)$ ($0\leq t\leq 1$) of *-homomorphisms
from $A$ to $B$ connecting $\phi_0$ and $\phi_1$ with $\|\phi_t(x_j)-\phi_0(x_j)\|<\epsilon$
for all $j$ for any $t\in[0,1]$.  (The $\delta$ depends on $\epsilon$, $A$, and the set $\G$ of generators, but not on the $B$, $\phi_0$, $\phi_1$.)
\end{Cor}

In fact, for any $\epsilon>0$, a $\delta$ that works for \ref{CloseExtend} also works for \ref{CloseHomSmall}.

\begin{proof}
Choose $\delta>0$ as in \ref{CloseExtend} for the given $\epsilon$.  Let $\tilde B=C([0,1],B)$, $I=C_0((0,1),B)$ the ideal of
elements of $\tilde B$ vanishing at 0 and 1.  Then $\tilde B /I\cong B\oplus B$.  Define $\phi,\psi:A\to \tilde B /I$ by
$\phi(x)=(\phi_0(x),\phi_0(x))$ and $\psi(x)=(\phi_0(x),\phi_1(x))$.  Then $\phi$ and $\psi$ satisfy the hypotheses of \ref{CloseExtend},
and $\phi$ lifts to $\tilde B$ as a constant function, so $\psi$ also lifts, and the lift satisfies the conclusion of \ref{CloseExtend}.
\end{proof}

\begin{Cor}\label{CloseHom}
Let $A$ be a semiprojective C*-algebra generated by a finite or countable set $\G=\{x_1,x_2,\dots\}$ with $\lim_{j\to\infty}\|x_j\|=0$ if
$\G$ is infinite.   Then there is a $\delta>0$
such that, whenever $B$ is a C*-algebra, $\phi_0$ and $\phi_1$ *-homomorphisms from $A$ to $B$
with $\|\phi_0(x_j)-\phi_1(x_j)\|<\delta$ for all $j$, then $\phi_0$ and $\phi_1$ are homotopic.
(The $\delta$ depends on $A$ and the set $\G$ of generators, but not on the $B$, $\phi_0$, $\phi_1$.)
\end{Cor}

\begin{proof}
Fix any $\epsilon>0$, say $\epsilon=1$, and apply \ref{CloseHomSmall}.
\end{proof}

In the proofs of the commutative versions of these results, a metric is fixed on the space and the $\delta$ depends on $\epsilon$
and the choice of metric.  Fixing a set of generators can be regarded as an analog of fixing a metric in our setting.

\section{The Homotopy Lifting Theorem}

We can now state and prove the C*-analog of the Borsuk Homotopy Extension Theorem.  When arrows are turned around
for the C*-algebra setting, extension problems become lifting problems.

\begin{Thm} {\sc [Homotopy Lifting Theorem]}\label{HomExt}
Let $A$ be a semiprojective C*-algebra, $B$ a C*-algebra, $I$ a closed ideal of $B$, $(\phi_t)$ ($0\leq t\leq1$) a point-norm continuous path of *-homomorphisms
from $A$ to $B/I$.  Suppose $\phi_0$ lifts to a *-homomorphism $\bar\phi_0:A\to B$, i.e.\ $\pi_I\circ\bar\phi_0=\phi_0$.
Then there is a point-norm continuous path $(\bar\phi_t)$ ($0\leq t\leq1$) of *-homomorphisms from $A$ to $B$ beginning at $\bar\phi_0$
such that $\bar\phi_t$ is a lifting of $\phi_t$ for each $t$, i.e.\ the entire homotopy lifts.  In particular,
$\phi_1$ lifts to a *-homomorphism from $A$ to $B$.
\end{Thm}

\begin{proof}
Let $\G=\{x_1,x_2,\dots\}$ be a countable set of generators for $A$, with $\|x_j\|\to0$ (by definition, a semiprojective
C*-algebra is separable, hence countably generated).  Fix $\epsilon>0$, say $\epsilon=1$, and fix $\delta>0$
satisfying the conclusion of \ref{CloseExtend} for $\epsilon$, $A$, $\G$.  Choose a finite partition
$0=t_0<t_1<t_2<\cdots<t_m=1$ such that $\|\phi_s(x_j)-\phi_t(x_j)\|<\delta$ for all $j$ whenever $s,t\in[t_{i-1},t_i]$ for any $i$.
There is such a partition since one only needs to consider finitely many $x_j$, the condition being automatic for
any $x_j$ with $\|x_j\|<\frac{\delta}{2}$; cf.\ the last part of the proof of \ref{CloseExtend}.

Begin with $[0,t_1]$.  Let $\tilde B=C([0,t_1],B)$, and $J$ the ideal of $\tilde B$ consisting of functions $f:[0,t_1]\to I$
with $f(0)=0$.  Then
$$\tilde B /J\cong C([0,t_1],B/I)\oplus_{\pi_I}B=\{(f,b)\in C([0,t_1],B/I)\oplus B:f(0)=\pi_I(b)\}\ .$$
Define homomorphisms $\phi,\psi:A\to \tilde B/J$ by setting $\phi(x)=(f_x,\bar\phi_0(x))$, where $f_x(t)=\phi_0(x)$ for all $t$,
and $\psi(x)=(g_x,\bar\phi_0(x))$, where $g_x(t)=\phi_t(x)$ for all $t$.  We then have
$$\|\phi(x_j)-\psi(x_j)\|<\delta$$
for all $j$.  Since $\phi$ lifts to a *-homomorphism from $A$ to $\tilde B$ (e.g.\ by the constant function $\bar\phi_0$),
$\psi$ also lifts, defining a continuous path of lifts $(\bar\phi_t)$ of the $\phi_t$ for $0\leq t\leq t_1$.

Now repeat the process on $[t_1,t_2]$, using the lift $\bar\phi_{t_1}$ as the starting point, and continue through
all the intervals.  After a finite number of steps the entire homotopy is lifted.
\end{proof}

\begin{Cor}
Let $A$ be a semiprojective C*-algebra, $B$ a C*-algebra, $I$ a closed ideal of $B$, $\phi$ a *-homomorphism from
$A$ to $B/I$.  If $\phi$ is homotopic to a *-homomorphism from $A$ to $B/I$ which lifts to $B$, then $\phi$ lifts to $B$.
\end{Cor}

This corollary gives an arguably simpler proof than the one in \cite{ThielInductive} that a contractible semiprojective C*-algebra
is projective, since the zero homomorphism always lifts from any quotient.  (The result in  
\cite{ThielInductive} is slightly more general).

\section{$\ell$-Open and $\ell$-Closed C*-Algebras}

In this section, all C*-algebras will be assumed {\em separable}.  We will use $\cC$ to denote a category of separable
C*-algebras and *-homomorphisms, e.g.\ the category of all separable C*-algebras and *-homomorphisms, the category of
separable unital C*-algebras and unital *-homomorphisms, or the category of separable unital commutative C*-algebras and unital
*-homomorphisms.

If $A$ and $B$ are C*-algebras, denote by $\Hom(A,B)$ the set of *-homomorphisms from $A$ to $B$, endowed with the
point-norm topology.  $\Hom(A,B)$ is separable and metrizable.  If $A$ and $B$ are unital,
let $\Hom_1(A,B)$ be the set of unital *-homomorphisms from $A$ to $B$.  $\Hom_1(A,B)$ is a clopen subset of $\Hom(A,B)$
(since a projection close to the identity in a C*-algebra is equal to the identity).

If $A=C(X)$ and $B=C(Y)$,
then $\Hom_1(A,B)$ is naturally homeomorphic to $X^Y$, the set of continuous functions from $Y$ to $X$, endowed
with the topology of uniform convergence (with respect to any fixed metric
on $X$, or with respect to the unique uniform structure on $X$ compatible with its topology).

More generally, if $\cC$
is a category of C*-algebras, denote by $\Hom_{\cC}(A,B)$ the morphisms in $\cC$, with the point-norm
topology (i.e.\ the subspace topology from $\Hom(A,B)$).

\smallskip

If $\cC$ is a category of C*-algebras, $A,B\in\cC$, and $I$ is a closed ideal of $B$ compatible with $\cC$ (i.e.\
$B/I\in\cC$ and the quotient map $\pi_I$ is a morphism in $\cC$; this is automatic in the three categories above), denote by
$\Hom_{\cC}(A,B,I)$ the set of $\cC$-morphisms from $A$ to $B/I$ which lift to $\cC$-morphisms from $A$ to $B$.
$\Hom_{\cC}(A,B,I)$ is a subset of $\Hom_{\cC}(A,B/I)$.

If $\cC$ is the category of separable unital commutative C*-algebras and $A=C(X)$, $B=C(Y)$, with $X$, $Y$ compact
metrizable spaces, $I$ corresponds to a closed subset $Z$ of $Y$ and $B/I\cong C(Z)$; then $\Hom_{\cC}(A,B,I)$ is
the subset $\XYZ$ of $X^Z$ consisting of maps (continuous functions) from $Z$ to $X$ which extend to maps from $Y$ to $X$.
See the companion article \cite{BlackadarExtending} for a discussion of this case.

Examples show that $\Hom_{\cC}(A,B,I)$ is neither open nor closed in $\Hom_{\cC}(A,B/I)$ in general (see \cite{BlackadarExtending} for commutative examples).
We seek conditions on $A$ insuring that $\Hom_{\cC}(A,B,I)$ is always open or closed in $\Hom_{\cC}(A,B/I)$ for any $B$ and $I$.

\begin{Def}
Let $\cC$ be a category, and $A\in\cC$.
\begin{enumerate}
\item[(i)]  $A$ is {\em $\ell$-open} (in $\cC$) if, for every pair $(B,I)$ in $\cC$, the set $\Hom_{\cC}(A,B,I)$ is open in $\Hom_{\cC}(A,B/I)$.
\item[(ii)]  $A$ is {\em $\ell$-closed} (in $\cC$) if, for every pair $(B,I)$ in $\cC$, the set $\Hom_{\cC}(A,B,I)$ is closed in $\Hom_{\cC}(A,B/I)$.
\end{enumerate}
If $\cC$ is the category of all separable C*-algebras, we just say $A$ is $\ell$-open [$\ell$-closed].
\end{Def}

If $\cC$ is the category of separable unital commutative C*-algebras and $A=C(X)$, then $A$ is $\ell$-open [$\ell$-closed] in $\cC$
if and only if $X$ is $e$-open [$e$-closed] in the sense of \cite{BlackadarExtending}.  (The $\ell$ and
$e$ stand for {\em liftable} and {\em extendible} respectively, the dual notions in the algebra and topology
contexts.)

\smallskip

The next result is an immediate corollary of \ref{CloseExtend}:

\begin{Cor}
Every semiprojective C*-algebra is both $\ell$-open and $\ell$-closed.
\end{Cor}

\begin{proof}
One only needs to observe that if $\G=\{x_1,x_2,\dots\}$ is a set of generators for $A$ with $\|x_j\|\to0$, and $\phi_n$, $\phi$
*-homomorphisms from $A$ to a C*-algebra $B$, then $\phi_n\to\phi$ in the point-norm topology if and only if, for every $\epsilon>0$,
there is an $n$ such that $\|\phi_k(x_j)-\phi(x_j)\|<\epsilon$ for all $j$, for all $k>n$. Apply \ref{CloseExtend}.
\end{proof}

If $\cC$ is the category of separable unital commutative C*-algebras and $A=C(X)$, then it is shown in \cite{BlackadarExtending}
that $A$ is $\ell$-open in $\cC$ if and only if $X$ is an ANR, at least if $A$ is finitely generated (equivalently, if $X$ is
finite-dimensional).  Recall that $A$ is semiprojective in $\cC$ if and only if $X$ is an ANR \cite{BlackadarShape}.
Thus it is reasonable to conjecture that a C*-algebra is $\ell$-open if and only if it is semiprojective, at least if
it is finitely generated.

Although there is no obvious direct proof that an $\ell$-open C*-algebra is $\ell$-closed,
I do not know an example of a C*-algebra which is $\ell$-open but not $\ell$-closed, and I conjecture that none exist.
There are $\ell$-closed C*-algebras which are not $\ell$-open, as example \ref{FreeProdC} shows.
I do not have a good idea how to characterize $\ell$-closed C*-algebras.

\smallskip

We conclude with some examples of C*-algebras which are not $\ell$-open.

\begin{Ex}\label{FreeProdC}
(A C*-algebra which is $\ell$-closed but not $\ell$-open.)
Let $A$ be the universal C*-algebra generated by a sequence of projections $\{p_1,p_2,\dots\}$, i.e.\ $A$ is the full free product of a
countable number of copies of $\C$.  Then $A$ is not $\ell$-open: let $B=C([0,1])$, $I=C_0((0,1))$.
$B/I\cong\C\oplus\C$.  Define $\phi_n:A\to B/I$ by $\phi_n(p_k)=(0,0)$ if $k\leq n$, $\phi_n(p_k)=(0,1)$ if $k>n$.
Then $\phi_n$ converges point-norm to the zero homomorphism from $A$ to $B/I$, which obviously lifts to $B$, but no
$\phi_n$ lifts to $B$.  (This shows that $A$ is not semiprojective, which can also be shown by a direct argument.)

$A$ is, however, $\ell$-closed.  Let $B$ be a C*-algebra and $I$ a closed ideal of $B$.  A sequence $(\phi_n)$
of homomorphisms from $A$ to $B/I$ converging point-norm to $\phi$ defines a set $q_k^{(n)}=\phi_n(p_k),q_k=\phi(p_k)$ of projections
in $B/I$ such that $q_k^{(n)}\to q_k$ for all $k$.  If each $\phi_n$ is liftable to $B$, i.e.\ each $q_k^{(n)}$
is liftable to a projection in $B$, it then follows from the semiprojectivity of $\C$ and \ref{CloseExtend}
that each $q_k$ is also liftable to a projection in $B$, i.e.\ $\phi$ is liftable to $B$.

\smallskip

A similar argument shows that a full free product of a sequence of semiprojective C*-algebras is always $\ell$-closed, although
it is not semiprojective unless all but finitely many of the C*-algebras are projective; does the latter condition also characterize when the free product is $\ell$-open?  (This seems likely.)
\end{Ex}

\begin{Ex}\label{FreeGp}
Let $A=C^*(\bbF_\infty)$, the full group C*-algebra of the free group on infinitely many generators, i.e.\ the universal
C*-algebra generated by a sequence of unitaries $\{u_1,u_2,\dots\}$.  It is known that
$A$ is not semiprojective (\cite{BlackadarSemiprojectivity}, \cite[II.8.3.16(vii)]{BlackadarOperator}).  To directly
show $A$ is not $\ell$-open,
let $S$ be the unilateral shift on $\cH=\ell^2$, and $B=\cT$ the C*-subalgebra of $\B(\cH)$ generated by $S$ (the Toeplitz algebra).
Then $B$ contains $I=\cK(\cH)$, and $B/I\cong C(\T)$.  Let $s$ be the image of $S$ in $B/I$.  It is well known that
$s$ has no normal preimage in $B$, in fact no normal preimage in $\B(\cH)$, cf.\ \cite{BrownDFExtensions}; in particular,
it has no unitary preimage in $B$.  Define $\phi_n:A\to B/I$ by setting $\phi_n(u_k)=1$ for $k\leq n$ and $\phi_n(u_k)=s$
for $k>n$.  Then $\phi_n\to\phi$ point-norm, where $\phi(u_k)=1$ for all $k$, and $\phi$ lifts to $B$, but no $\phi_n$ lifts.

An argument similar to the one in \ref{FreeProdC}, using semiprojectivity of $C(\T)$, shows that $A$ is $\ell$-closed in the category of separable unital C*-algebras
and unital *-homomorphisms.  (More generally, a full unital free product of a sequence of unital semiprojective C*-algebras is $\ell$-closed
in the unital category.)  However, it seems like a difficult and delicate question whether $A$ is $\ell$-closed (in the
general category).  For a sequence of homomorphisms from $A$ to $B/I$ defines a convergent sequence $(q_n)$ of projections in $B/I$ (the images
of the identity of $A$) and a sequence of unitaries in $q_n(B/I)q_n$ for each $n$.  The $q_n$ and the unitaries must be
lifted in a compatible way to obtain a lifting of the limit projection and unitaries.  So:

\smallskip

Is $A$ $\ell$-closed?
\end{Ex}

\begin{Ex}
Let $A$ be the universal C*-algebra generated by a normal element $x$ of norm $\leq1$.  Then $A\cong C_0(\bbD\setminus\{(0,0)\})$,
the functions vanishing at $(0,0)$ on the closed unit disk $\bbD$ in $\R^2$.  To show that $A$ is not $\ell$-open,
let $B$, $I$, $S$, $s$ be as in \ref{FreeGp}.  Define $\phi_n:A\to B/I$
by sending $x$ to $\frac{1}{n}s$.  Then $(\phi_n)$ converges in the point-norm
topology to the zero homomorphism, which obviously lifts to $B$.  But no $\phi_n$ lifts.

Showing that $A$ is $\ell$-closed is the same as solving (positively) the following problem:  if $(y_n)$ is a convergent sequence
of normal elements in a quotient $B/I$ with limit $y$, and each $y_n$ lifts to a normal element in $B$, does $y$ also lift
to a normal element?  This appears to be unknown.

If this argument works, it can be slightly modified to show that the unitization $C(\bbD)$ is $\ell$-closed but not $\ell$-open.
In fact, it seems reasonable that if $X$ is any ANR, then $C(X)$ is $\ell$-closed, but it is $\ell$-open if and only if
$C(X)$ is semiprojective, i.e.\ if and only if $dim(X)\leq1$ \cite{SorensenTCharacterization}.
\end{Ex}

\begin{Ex}
Consider the C*-algebras $c$ of convergent sequences of complex numbers and $c_0$ of sequences of complex numbers converging to 0.

\smallskip

To show they are not $\ell$-open, let $B=C([0,1])$ and $I$ the ideal of functions which vanish at $\frac{1}{n}$ for all $n$
(and hence of course also at 0).  Then $B/I\cong c$.  Define $\phi_n:c\to B/I$ by setting $[\phi_n(x)](1/k)=\alpha_k$ if $k>n$,
$[\phi_n(x)](1/k)=\alpha$ if $k\leq n$, $[\phi_n(x)](0)=\alpha$, for $x=(\alpha_1,\alpha_2,\dots)\in c$ with $\alpha_n\to\alpha$.
Then $\phi_n\to\phi$ in the point-norm topology, where $\phi(x)$ is the constant function with value $\alpha$.  Then $\phi$
lifts to $B$, but no $\phi_n$ lifts to $B$ since $B$ has no nontrivial projections.  The restrictions of $\phi_n$, $\phi$
to $c_0$ work the same way.

\smallskip

The question whether $c$ and $c_0$ are $\ell$-closed is much more involved than in the commutative case.  It is relatively
easy to show they are $\ell$-closed in the commutative category (cf.\ \cite{BlackadarExtending});
the commutative case is simpler since
\begin{enumerate}
\item[(i)]  Close projections in a commutative C*-algebra are actually equal.
\item[(ii)]  A product of two commuting projections is a projection.  In particular, if $q$ is a projection in a quotient $B/I$, with $B$ commutative,
and $p_1,p_2$ are two projection lifts to $B$, then $p=p_1p_2$ is also a projection lift to $B$ with $p\leq p_1$, $p\leq p_2$.
Nothing like this is true for general noncommutative $B$.
\end{enumerate}

A *-homomorphism from $c_0$ to a C*-algebra $B$ is effectively the same thing as a specification of a sequence of mutually
orthogonal projections $(p_k)$ in $B$ (some of which may be 0): such a sequence defines a homomorphism $\phi$ by
$$\phi((\alpha_1,\alpha_2,\dots))=\sum_{k=1}^\infty \alpha_kp_k$$
(the sum converges in $B$ since $\alpha_k\to0$).  For a homomorphism from $c$ to $B$, we additionally need a projection $p$
such that $p_k\leq p$ for all $n$: the homomorphism corresponding to such a set of projections is defined by
$$\phi((\alpha_1,\alpha_2,\dots))=\alpha p+\sum_{k=1}^\infty (\alpha_k-\alpha)p_k$$
where $\alpha=\lim_{k\to\infty}\alpha_k$.  If $(\phi_n)$ is a sequence of homomorphisms corresponding to $(p_k^{(n)},p^{(n)})$,
and $\phi$ is another homomorphism corresponding to $(p_k,p)$, then $\phi_n\to\phi$ in the point-norm topology if and only if
$\lim_{n\to\infty}p_k^{(n)}=p_k$ for each $k$ and $\lim_{n\to\infty}p^{(n)}=p$.

Now suppose $B$ is a C*-algebra and $I$ a closed ideal of $B$, and $\phi_n,\phi:c\to B/I$ with $\phi_n\to\phi$.
Let $\phi_n$ correspond to $(q_k^{(n)},q^{(n)})$ and $\phi$ to $(q_k,q)$.  Suppose each $q_k^{(n)}$ lifts to a projection
in $B$.  We need to find projections $(p_k,p)$ in $B$ with the $p_k$ mutually orthogonal, $p_k\leq p$ for all $k$, $\pi_I(p_k)=q_k$ for all $k$,
and $\pi_I(p)=q$.  It seems technically difficult, if not impossible, to show that this can be done.  So:

\smallskip

Are $c$ and $c_0$ $\ell$-closed?




\end{Ex}

\bibliography{sprojhomref}
\bibliographystyle{alpha}

\end{document}